\documentclass[12pt,a4paper]{article}
\usepackage[english]{babel}

\usepackage{pstricks}
\usepackage{pst-node}
\usepackage{amsmath}
\usepackage{amssymb}
\usepackage{graphicx}
\usepackage{enumerate}
\usepackage{color}
\usepackage[text={15cm,24cm}]{geometry}

\newenvironment{proof}{{\bf Proof}:\ }%
   {~\ \hfill $\Box$\vspace{0,5cm}}

    {~\ \hfill$\Diamond$\vspace{0,5cm}}

    {~\ \hfill$\Box$\vspace{0,5cm}}

\newtheorem{prop}{Property}[section]
\newtheorem{theorem}{Theorem}[section]

\newtheorem{rmk}{Remark}[section]

\newtheorem{lemma}[theorem]{Lemma}
\newtheorem{proposition}[theorem]{Proposition}

\newtheorem{coro}[theorem]{Corollary}

\newtheorem{fact}{Fact}[section]

\newrgbcolor{lightlightlightgray}{0.9 0.9 0.9}

\numberwithin{equation}{section}

\begin{document}

\title{On the vertices belonging to all, some, none minimum dominating set}
\author{
Valentin Bouquet\footnotemark[1]
\and
Fran\c cois Delbot\footnotemark[2]
\and
Christophe\ Picouleau \footnotemark[3]
}
\date{\today}

\footnotetext[1]{ \noindent
Université Paris-Nanterre, Paris (France). Email: {\tt
valentin.bouquet@parisnanterre.fr
}
}

\footnotetext[2]{ \noindent
Sorbonne Université, Laboratoire d'Informatique de Paris 6 (LIP6), Paris (France). Email: {\tt
francois.delbot@lip6.fr
}}

\footnotetext[3]{ \noindent
Conservatoire National des Arts et M\'etiers, CEDRIC laboratory, Paris (France). Email: {\tt
christophe.picouleau@cnam.fr}
}

\graphicspath{{.}{graphics/}}

%\begin{document}
\maketitle
\begin{abstract}
We characterize the vertices belonging to all minimum dominating sets, to some minimum dominating sets but not all, and to no minimum dominating set. We refine this characterization for some well studied sub-classes of graphs: chordal, claw-free, triangle-free. Also we exhibit some graphs answering to some open questions of the literature on minimum dominating sets.

 \vspace{0.2cm}
 \end{abstract}

%\newpage
\parindent=0cm
%----------------------------------------------------------------------------------------------------------------------------------
\section{Introduction}
We will only be concerned with simple undirected graphs $G=(V,E)$. The reader is referred to \cite{Bondy}  for definitions and notations in graph theory. An element $ab\in E$ is called an {\it edge}, if $ab\not\in E$ then $ab$ is called a {\it non-edge}. For a vertex $v\in V$ let us denote by $N(v)$ its neighborhood, $N[v]=N(v)\cup\{v\}$ its closed neighborhood. $N_k(v)$ is the set of vertices at distance exactly $k$ of $v$. Hence $N(v)=N_1(v)$ and $N[v]=N_0(v)\cup N_1(v)$. A vertex $v$ is {\it isolated} if $N(v)=\emptyset$. A vertex $v$ is {\it universal} if $N[v]=V$. Two distinct vertices $u,v$ are {\it twins} if $N(v)=N(u)$, are {\it false twins} if $N[v]=N[u]$.

For a subset $S\subseteq V$, we let $G[S]$ denote the subgraph of $G$ {\it induced} by $S$, which has vertex set~$S$ and edge set $\{uv\in E\; |\; u,v\in S\}$. Moreover, for a vertex $v\in V$, we write $G-v=G[V\setminus \{v\}]$ and for a subset $V'\subseteq V$ we write $G-V'=G[V\setminus V']$.
For a set $\{H_1,\ldots,H_p\}$ of graphs, $G$ is {\it $(H_1,\ldots,H_p)$-free} if $G$ has no induced subgraph isomorphic to a graph in $\{H_1,\ldots,H_p\}$; if $p=1$ we may write $H_1$-free instead of $(H_1)$-free. For two vertex disjoint induced subgraphs $G[A],G[B]$ of $G$, $G[A]$ is {\it complete} to $G[B]$ if $ab$ is an edge for any $a\in A$ and $b\in B$, $G[A]$ is {\it anti-complete} to $G[B]$ if $ab$ is an non-edge for any $a\in A$ and $b\in B$.

A  set $S\subseteq V$ is called a {\it stable set} or an {\it independent set} if  any pairwise distinct vertices $u,v\in S$ are non adjacent. The maximum cardinality of an independent set in $G$ is denoted by $\alpha(G)$. A  set $S\subseteq V$ is called a {\it clique}  if  any pairwise distinct vertices $u,v\in S$ are  adjacent. When $G[V]$ is a clique then  $G$ is a {\it complete graph}. $K_p,p\ge 1,$ is the clique or the complete graph on $p$ vertices.

For $n\geq 1$, the graph $P_n=u_1-u_2-\cdots-u_n$ denotes the {\it cordless path} on $n$ vertices, that is, $V({P_n})=\{u_1,\ldots,u_n\}$ and $E({P_n})=\{u_iu_{i+1}\; |\; 1\leq i\leq n-1\}$.
For $n\geq 3$, the graph $C_n$ denotes the {\it cordless cycle} on $n$ vertices, that is,  $V({C_n})=\{u_1,\ldots,u_n\}$ and $E({C_n})=\{u_iu_{i+1}\; |\; 1\leq i\leq n-1\}\cup \{u_nu_1\}$. For $n\ge 4$, $C_n$ is called a {\it hole}. The graph $C_3=K_3$ is also called the {\it triangle}.
The {\it claw} $K_{1,3}$ is the 4-vertex star, that is, the graph with vertices $u$, $v_1$, $v_2$, $v_3$ and edges $uv_1$, $uv_2$, $uv_3$. The {\it diamond}  is the 4-vertex complete graph $K_4$ minus an edge.The {\it net} is the graph with six vertices $u_1,u_2,u_3$, $v_1$, $v_2$, $v_3$ and edges $u_1u_2,u_2u_3,u_1u_3,u_1v_1,u_2v_2,u_3v_3$. The {\it bull} is the graph with five vertices $u_1,u_2,u_3$, $v_1$, $v_2$ and edges $u_1u_2,u_2u_3,u_1u_3,u_1v_1,u_2v_2$. The {\it paw} is the graph with four vertices $u_1,u_2,u_3$, $v_1$ and edges $u_1u_2,u_2u_3,u_1u_3,u_1v_1$.  (see Figure \ref{defgraphs}).

\begin{figure}[htbp]
\begin{center}
\includegraphics[width=10cm, height=5cm, keepaspectratio=true]{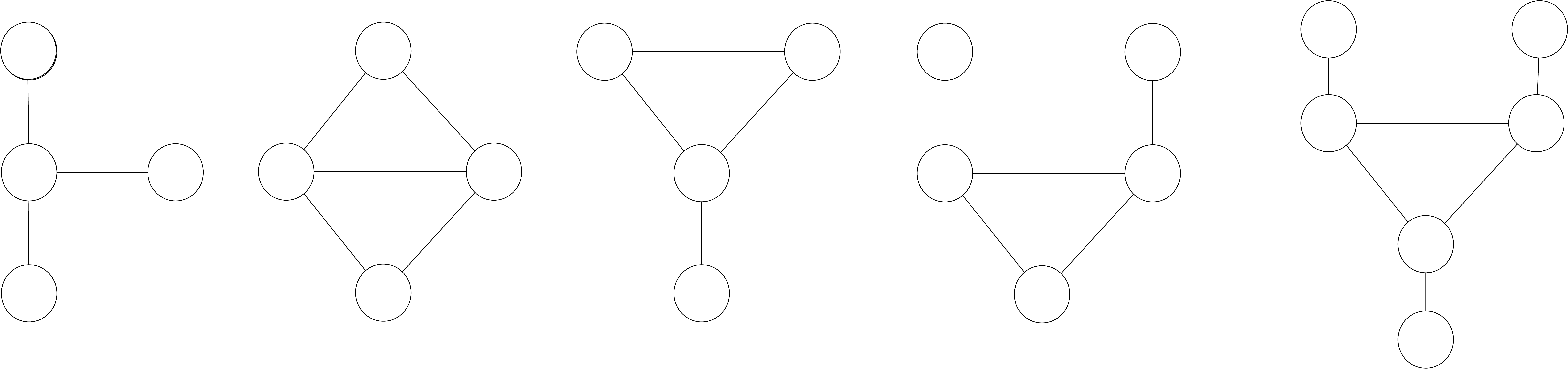}
\end{center}
\caption{The claw, the diamond, the paw, the bull, the net.}
\label{defgraphs}
\end{figure}

A  set $S\subseteq V$ is called a {\it dominating set} if every vertex $v\in V$ is either an element of $S$ or is adjacent to an element of $S$. The minimum cardinality of a dominating set in $G$ is denoted by $\gamma(G)$ and called the {\it dominating number} of $G$. A dominating set $S$ with $\vert S\vert=\gamma(G)$ is called a {\it Minimum Dominating Set}, a mds for short. Following \cite{DomBook} a mds is also called a $\gamma$-set.  If $S\subset V$ is both a dominating and an independent set then $S$ is an {\it independent dominating set}. The minimum cardinality of an independant dominating set in $G$ is denoted by $i(G)$. Clearly we have $\gamma(G)\le i(G)\le \alpha(G)$. Note that a minimum independent dominating set is a {\it minimum maximal independent set}.

In a same flavour than Boros et al. in \cite{Boros}
for the maximum stable set, let $\Omega(G)$ denote the family of all Minimum Dominating Sets of the graph $G$. Let $core(G)=\bigcap\{S: S\in\Omega(G)\}$ be the set of vertices belonging to all $\gamma$-sets.  Similarly, let us denote $corona(G)=\bigcup\{S: S\in \Omega(G)\}$ as the set of vertices belonging to some $\gamma$-set of $G$. Let us denote by $anticore(G)=V-corona(G)$ the set of vertices not belonging to any mds of $G$.

We are interested in characterizing the vertices $v\in core(G)$, the vertices $v\in corona(G)$, the vertices $v\in anticore(G)$.

 Let $S\subset V$ and $u\in S$. We say that a vertex $v$ is a {\it private neighbor} of $u$, with respect to $S$, if $N[v]\cap S=\{u\}$. We define the {\it private neighbor set} of $u$, with respect to $S$, to be $pn[u,S]=\{v: N[v]\cap S=\{u\}\}$. As remarked in \cite{DomBook} page 18, if $S$ is a $\gamma$-set then for every $u\in S,pn[u,S]\ne\emptyset$.\\

The following partition of the vertex set $V$ is defined in \cite{DomBook} page 136:

\begin{itemize}
\item $V^0=\{v\in V: \gamma(G-v)=\gamma(G)\}$;
\item $V^+=\{v\in V: \gamma(G-v)>\gamma(G)\}$;
\item $V^-=\{v\in V: \gamma(G-v)<\gamma(G)\}$.
\end{itemize}

The following characterizations are given in \cite{DomBook} page 137:

\begin{theorem}\label{V+}
A vertex $v\in V^+$ iff
\begin{enumerate}
\item $v$ is not an isolate vertex and $v\in core(G)$, and
\item no subset $S\subseteq V-N[v]$ with cardinality $\gamma(G)$ dominates $G-v$.
\end{enumerate}
\end{theorem}

\begin{theorem}\label{V-}
A vertex $v\in V^-$ iff $pn[v,S]=\{v\}$ for some $\gamma$-set containing $v$.
\end{theorem}
\begin{rmk}\label{ACV-}
From Theorem \ref{V-} no vertex $v\in anticore(G)$ can belong to $V^-$.
\end{rmk}

The article is organized as follows. In the next section we give the characterizations for a vertex $v$ to be a member of  either $core(G)$ or $corona(G)-core(G)$ or $anticore(G)$. Then for some subclasses of graphs we show that no vertex can be in $core(G)\cap V^0$. Then we answer to some open questions given in \cite{Samodivkin} and, with the same flavor, we give some graphs with a particular partition of their vertex set. 

\section{The characterizations}
In this section we use Theorems \ref{V+} and \ref{V-} to characterize the membership of a given vertex: $v\in core(G)$ or $v\in corona(G)-core(G)$ or $v\in anticore(G)$. Also we give algorithmic issues of our characterization.

Given $G=(V,E)$ and  $v\in V$ we define the graph $G_v+u=(V',E')$ as follows: $V'=V\cup\{u\}$ and $E'=E\cup \{uv\}$.

\begin{rmk}\label{G+u}
One can note that any mds of $G_v+u$ contains either $u$ or $v$. Moreover if there exists a mds of $G_v+u$ that contains $u$ then there is another one that contains $v$ (note that the converse is not necessarily true).
\end{rmk}
\begin{theorem}\label{vinanticore}
    $v \in anticore(G)$ iff $\gamma(G_v+u)= \gamma(G)+1$.\end{theorem}
\begin{proof}
Let $v\in anticore(G)$. From Remark \ref{G+u} any mds $S$ of $G_v+u$ contains either $u$ or $v$. It follows that $S$ has one vertex more than any mds of $G$. Now let $v\not\in anticore(G)$. Suppose that $\gamma(G_v+u)= \gamma(G)+1$. It exists $S$ a mds of $G$ containing $v$. Yet $S$ is a dominating set of $G_v+u$, a contradiction.
\end{proof}

\begin{lemma}\label{V-isol}
$v\in core(G)\cap V^-$ if and only if $v$ is an isolated vertex.
\end{lemma}
\begin{proof}
Clearly if $v$ is isolated then $v\in V^-$ and $v\in core(G)$. Conversely let $v\in V^-$. If $v$ is isolated then any dominating set has to contain $v$. So $v\in core(G)\cap V^-$.  Suppose that $v$ is not isolated. From Theorem \ref{V-} it exists a mds $S, v\in S,$ such that $pn[v,S]=\{v\}$.  For any $w\in N(v)$,  $\vert N[w]\cap S\vert>1$ and $S'=S-\{v\}\cup\{w\}$ is a mds.  So $v\not\in core(G)$.
\end{proof}

\begin{theorem}\label{vincore}
$v\in core(G)$ iff either
\begin{enumerate}
\item $v$ is isolated or
\item $\gamma(G-v)>\gamma(G)$ or
\item $\gamma(G-v)=\gamma(G)$ and every subset $S,\vert S\vert=\gamma(G),$ that dominates $G-v$ is such that $S\cap N[v]=\emptyset$.
\end{enumerate}
\end{theorem}

\begin{proof}
$\Leftarrow$: if $v$ is isolated then $v\in core(G)$. If $\gamma(G-v)>\gamma(G)$ then $v\in V^+$ and from Theorem \ref{V+} $v\in core(G)$. Suppose 3.: $v\in V^0$ and  $v$ is not isolated. Let $S$ be a $\gamma$-set of $G$. If $v\not\in S$ then $S\cap N[v]\ne\emptyset$ but $S-\{v\}$ has cardinality $\gamma(G)$ and is a dominating set of $G-v$, a contradiction. So $v\in core(G)$

$\Rightarrow$: $v\in core(G)$. If $v\in V^-$ from Lemma \ref{V-isol} $v$ is isolated. If $v\in V^+$ then $\gamma(G-v)>\gamma(G)$. Now $v\in V^0$ that is  $\gamma(G-v)=\gamma(G)$. It exists $S,\vert S\vert =\gamma(G),$ that dominates $G-v$. If $S\cap N[v]\ne\emptyset$ then $S$ is a mds for $G$ with $v\not\in S$ and $v\not\in core(G)$. Thus every $S,\vert S\vert=\gamma(G),$ that dominates $G-v$ is such that $S\cap N[v]=\emptyset$.
\end{proof}

The figure \ref{Threecore} shows three types of vertices in $core(G)$ as stated by Theorem \ref{vincore}.\\
\begin{figure}[htbp]
\begin{center}
\includegraphics[width=9cm, height=5cm, keepaspectratio=true]{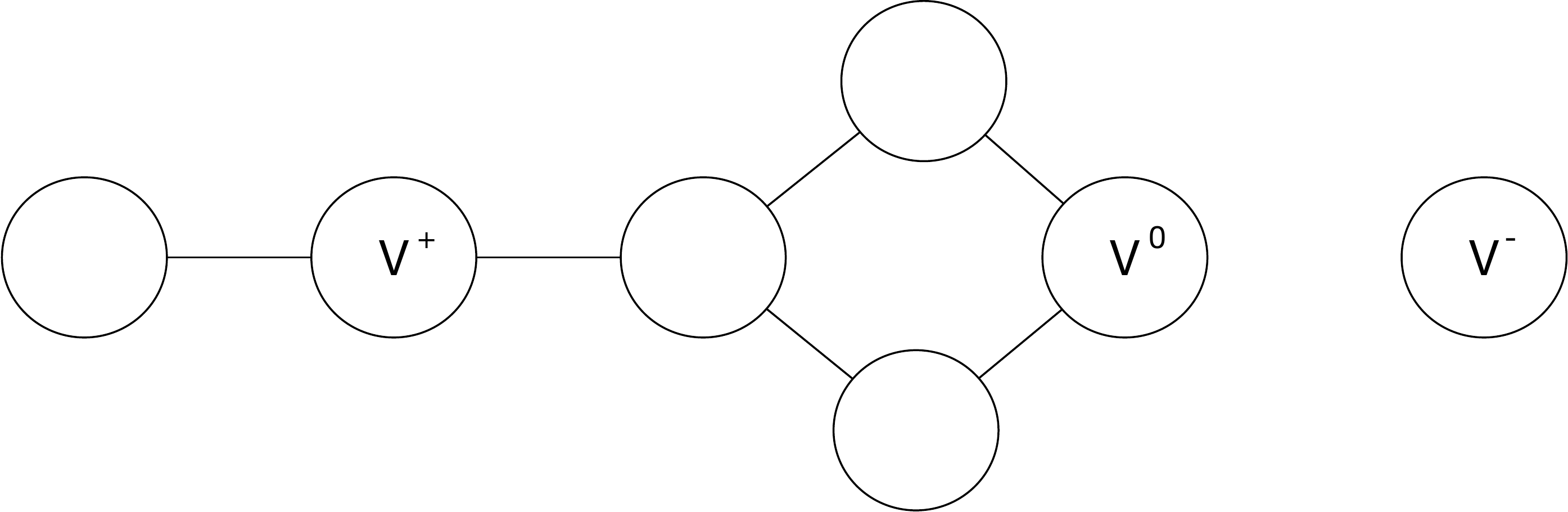}
\end{center}
\caption{$v^+\in core(G)\cap V^+$, $v^0\in core(G)\cap V^0$, $v^-\in core(G)\cap V^-$.}
\label{Threecore}
\end{figure}

From Theorem \ref{vinanticore}  the theorem \ref{vincore} can be stated as follow. This formulation will be more useful for a computational point of view.

\begin{theorem}\label{vincorealg}
$v\in core(G)$ iff either
\begin{enumerate}
\item $v$ is isolated or
\item $\gamma(G-v)>\gamma(G)$ or
\item $\gamma(G-v)=\gamma(G)$ and every neighbor $u\in N(v)$ is such that $u\in anticore(G-v)$.
\end{enumerate}
\end{theorem}

From Theorems \ref{vinanticore} and \ref{vincore} we obtain the following.
\begin{coro}\label{vinK}
$v\in corona(G)-core(G)$ iff either
\begin{enumerate}
\item $v\in V^-$ and $v$ is not isolated or
\item $v\in V^0$ and it exists $S,\vert S\vert=\gamma(G),$ that dominates $G-v$  such that $S\cap N[v]\not=\emptyset$ and $\gamma(G_v+u)= \gamma(G)$.
\end{enumerate}
\end{coro}

\begin{proposition}\label{mdsinP}
Let $\cal C$ be a class of graphs such that the minimum dominating set problem is polynomial time solvable for $G$ and $G_v+u$.  Given a graph $G=(V,E)\in \cal C$ and $v\in V$ the problem of deciding if either $v\in core(G)$ or  $v\in anticore(G)$ or  $v\in corona(G)-core(G)$ is in $P$.
\end{proposition}
\begin{proof}
Let $G\in \cal C$. To compute $k$ the minimum size of a dominating set can be done in polynomial time.  It follows from Theorem \ref{vinanticore} that checking if $v\in anticore(G)$ is polynomial. Clearly checking if $v$ satisfies item 1 or 2 of Theorem \ref{vincorealg} is polynomial, using Theorem \ref{vinanticore} checking  item 3 is polynomial. So deciding if either $v\in core(G)$ or  $v\in anticore(G)$ or  $v\in corona(G)-core(G)$ is in $P$.
\end{proof}

\section{Presence of vertices in $core(G)\cap V^0$ for some classes of graphs}\label{classes}
As stated by Theorem \ref{vincore} there are exactly three types of vertices in $core(G)$. Clearly when $G$ is connected and has at least two vertices an isolated vertex don't exists. In this section we show that for some classes of connected graphs with at least two vertices we have  $core(G)=V^+$. In order to draw the borderline between these classes and the classes where it may exists a vertex in $core(G)\cap V^0$, we exhibit some graphs where it exists a vertex $v\in core(G)\cap V^0$. These graphs are the {\it closest}, in some sense, to the classes we study. We also give some complexity results for some subclasses of chordal graphs. 

First we give some general properties we will use later. 

\begin{fact}\label{cliqueneigh}
Let $G=(V,E)$ be a connected graph with at least two vertices.  If $v\in V$ is such that $G[N[v]]$ is a clique then $v\not \in core(G)$. 
\end{fact}
\begin{proof}
Let $u\in N(v)$. If there exists $S$ a $\gamma$-set of $G$ then $(S-v)\cup{u}$ is another $\gamma$-set.
\end{proof}

\begin{lemma}\label{connect}
Let $v$ be a cut-vertex of $G$ such that $v\in core(G)$. If for any connected component $C$ of $G-v$ the vertices of $C\cap N(v)$ induces a clique then $v\in V^+$.
\end{lemma}
\begin{proof}
Let $v\in core(G)$. Suppose $v\in V^0$. From Theorem \ref{vincore}.3, $\gamma(G-v)=\gamma(G)$ and every subset $S,\vert S\vert=\gamma(G),$ that dominates $G-v$ is such that $S\cap N[v]=\emptyset$. Let $D$ be a mds of $G$. Let $C_1,\ldots,C_k,k\ge 2,$ be the connected components of $G-v$. Let $S_i=S\cap C_i,D_i=D\cap C_i,1\le i\le k$. Since $v\in D$ it exists $i$ such that $\vert D_i\vert <\vert S_i\vert$. W.l.o.g. let $i=1$. $D_1$ dominates $C_1-(C_1\cap N[v])$. If $D_1\cap N[v]\ne \emptyset$ then  $D_1\cup S_2\cup\cdots\cup S_k$ dominates $G$ (recall $C_1\cap N[v]$ is a clique) but $\vert D_1\vert+ \vert S_2\vert+\cdots\vert S_k\vert<\gamma(G)$. When $D_1\cap N[v]=\emptyset$ let $v_1\in C_1\cap N[v]$. Now $\vert D_1\cup\{v_1\}\vert\le \vert S_1\vert$, so  $D_1\cup\{v_1\}\cup S_2\cup\cdots\cup S_k$ is a $\gamma$-set of $G$ which not contains $v$.

From Theorem \ref{vincore} $v\not\in V^-$, thus $v\in V^+$.
\end{proof}

Here we correct a published error concerning the cut vertices.  
Proposition 5 in \cite{Bauer} states: {\it If a cutpoint $v$ of $G$ is in every minimum dominating set for $G$, then $\gamma(G-v)>\gamma(G)$}, that is a cut-vertex $v\in core(G)$ is such that $v\in V^+$. This is a mistake, the
Figure \ref{CutV0} shows a graph with a cut-vertex in  $core(G)\cap V^0$.
\begin{figure}[htbp]
\begin{center}
\includegraphics[width=8cm, height=4cm, keepaspectratio=true]{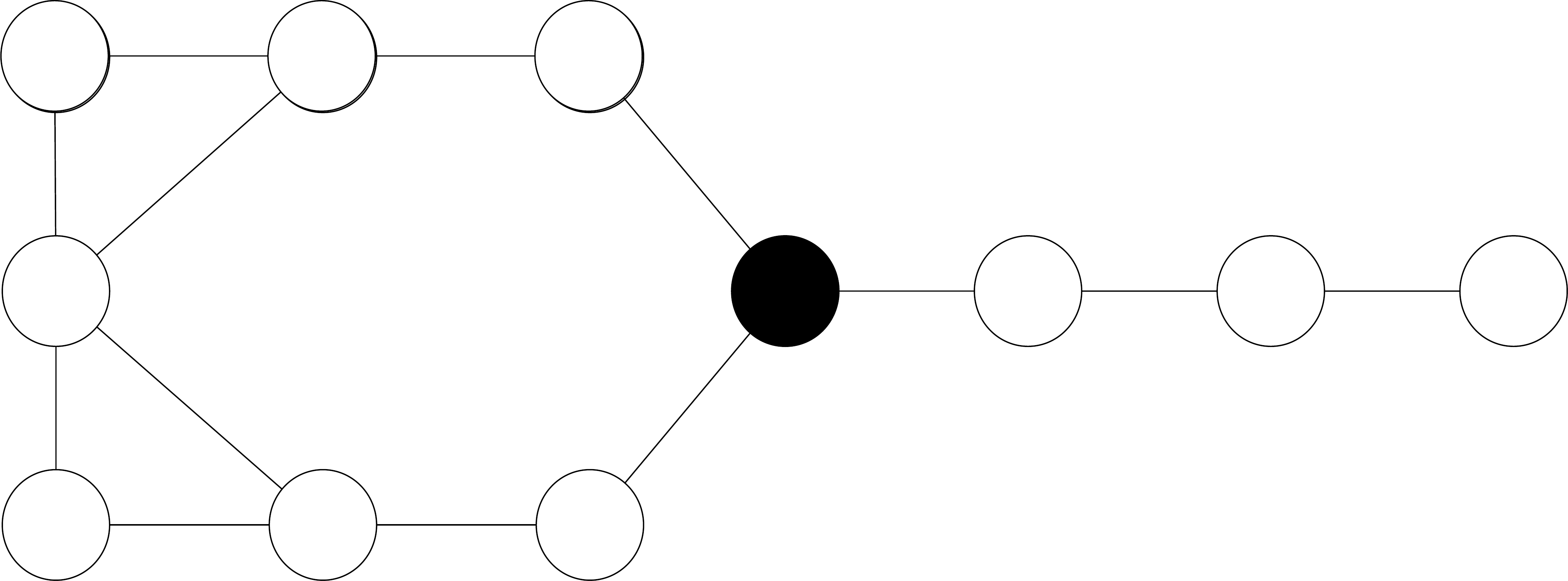}
\end{center}
\caption{The black cut-vertex is in $core(G)\cap V^0$.}
\label{CutV0}
\end{figure}

\begin{lemma}\label{CordV+}
Let $G$ be a connected graph with at least two vertices and $v\in core(G)$. Let $C_1,\ldots,C_k$ be the connected components of $G-N[v]$. Let $N_i$ be the set of neighbors of $v$ with a neighbor in $C_i,1\le i\le k$. If $G[N_i]$ is a clique then $v\in V^+$.
\end{lemma}
\begin{proof}
If there is no such component $C_i$, i.e. $G=N[v]$, then from Theorem \ref{vincore} $v\in V^+$.

When a component $C_i$ exists,  from Theorem \ref{vincore}, $\gamma(G-v)=\gamma(G)$ and every subset $S,\vert S\vert=\gamma(G),$ that dominates $G-v$ is such that $S\cap N[v]=\emptyset$. Let $D$ be a mds of $G$. Let $N_i$ be the set of neighbors of $v$ with a neighbor in $C_i,1\le i\le k$.  Let $S_i=S\cap (C_i\cup N_i),D_i=D\cap (C_i\cup N_i),1\le i\le k$.  

Since $v\in D$ it exists $i$ such that $\vert D_i\vert <\vert S_i\vert$. W.l.o.g. let $i=1$. $D_1$ dominates $C_1$. $S_2\cup\cdots\cup S_k$ dominates $C_2\cup\cdots\cup C_k$ and $N(v)-N_1$. Let $v_1\in N_1$. 
Thus $D_1\cup \{v_1\}\cup S_2\cup\cdots\cup S_k$ dominates $G$ since $N_1\cup\{v\}$ is a clique. So $v\in V^+$.
\end{proof}

\subsection{Chordal graphs}\label{chordal}
We recall that a graph is chordal if and only if every cycle on length at most four contains a chord. We show that in a chordal graph without isolated vertices if a vertex is in $core(G)$ then it is in $V^+$. Moreover we show that determining the status of a vertex can be done in linear time for trees and interval graphs, two subclasses of chordal graphs. We recall that the Minimum Dominating Set problem is $NP$-complete for chordal graphs \cite{Interval}.
\begin{prop}\label{chordcore}
Let $G=(V,E)$ be a connected chordal graph with at least two vertices. $v\in core(G)$ if and only if $\gamma(G-v)>\gamma(G)$.
\end{prop}

\begin{proof}
From Theorem \ref{vincore} if  $\gamma(G-v)>\gamma(G)$ then $v\in core(G)$. 
Now let $v\in core(G)$. Suppose that $v\in V^0$. From Theorem \ref{vincore} $N_2(v)\ne \emptyset$.

Suppose  it exists $N_i$ as defined in Lemma \ref{CordV+} that is not a clique: it exists $s,t\in N_i$ such that $st$ is a non-edge. So there is a path from $s$ to $t$ in $C _i$. Let $P$ be a such shortest path. Then $G[P\cup\{v\}]$ is a hole, contradiction. Thus each $N_i$ is a clique.
It follows from Lemma \ref{CordV+} $v\in V^+$.
\end{proof}

From Property \ref{chordcore}, Theorem \ref{vinanticore} and corollary \ref{vinK} we obtain the following.

\begin{theorem}\label{corechord}
Let $G=(V,E)$ be a chordal with at least two vertices.
\begin{itemize}
\item $v\in core(G)$ if and only if $v\in V^+$;
\item $v \in anticore(G)$ iff $\gamma(G_v+u)= \gamma(G)+1$;
\item $v\in corona(G)-core(G)$ iff $v\in V^-$ or $v\in V^0$ and $\gamma(G_v+u)= \gamma(G)$.
\end{itemize}
\end{theorem}

For the special case where $G$ is a tree we know the following.
Computing a minimum dominating set is linear for the class of trees \cite{Cockayne} and $G_v+u$ is a tree. So we deduces the following.
\begin{rmk}\label{treealg}
Let $G$ be a tree and $v$ be a vertex of $G$. Deciding if either $v\in core(G)$ or  $v\in anticore(G)$ or  $v\in corona(G)-core(G)$ can be done in  linear time.
\end{rmk}

Let us consider the case where $G$ is an interval graph.

\begin{proposition}
Let $G$ be an interval graph and $v$ be a vertex of $G$. Deciding if either $v\in core(G)$ or  $v\in anticore(G)$ or  $v\in corona(G)-core(G)$ can be computed in time $O(\vert V\vert+\vert E\vert)$.
\end{proposition}
\begin{proof}
Let $G=(V,E)$ be an interval graph and $v\in V$. The class of interval graph is a subclass of directed path graphs \cite{GraphCla}. It is easy to verify that $G_v+u$ is a directed path graph. From \cite{Interval} computing $\gamma(G')$ can be done in time $O(\vert V\vert+\vert E\vert)$ when $G'$ is a directed path graphs. Following Proposition \ref{mdsinP} determining the status of $v$ can be done in time $O(\vert V\vert+\vert E\vert)$.
\end{proof}

\subsection{Cographs}\label{cograph}
The class of cographs is also the class of $P_4$-free graphs, see  \cite{GraphCla}. If $G$ is cograph then it admits the following decomposition:

\begin{itemize}
\item  a vertex is a cograph;
\item if $G_1$ and $G_2$ are two cographs then $G_1+G_2$ is a cograph;
\item if $G_1$ and $G_2$ are two cographs then $G_1\times G_2$ is a cograph.
\end{itemize}
 
So if $G$ is a connected cograph with at least two vertices then $G=G_1\times G_2$ where $G_1,G_2$ are complete. 

\begin{proposition}\label{cographV0}
Let $G$ be a connected cograph with at least two vertices.  Then $0\le \vert core(G)\vert\le 1$. If $\vert core(G)\vert= 1$ then $core(G)=V^+$.
\end{proposition}
\begin{proof}
Let $G=G_1\times G_2$. Since $G$ is connected we have $\gamma(G)\le 2$ ($\{v_1,v_2\},v_1\in V_1,v_2\in V_2$ dominates $G$). If both $G_1$ and $G_2$ are not a clique we have $\gamma(G)= 2$ and any pair $\{v_1,v_2\}$ is a $\gamma$-set. So $core(G)=\emptyset$. If $G_1$ and $G_2$ are two cliques then $G$ is clique, so $\gamma(G)=1$ and $core(G)=\emptyset$. If $G_1$ is a clique and $G_2$ is not a clique then any vertex $v_1$ of $G_1$ is a mds, thus $\gamma(G)=1$. If $G_1$ contains more than one vertex $core(G)\ne\emptyset$. When $G_1=K_1$ then $\{v_1\}$ is  the unique $\gamma$-set, so $core(G)=\{v_1\}$. Since $G_2$ is not a clique $v_1\in V^+$.\end{proof}

This result is tight since Figure \ref{coreP5V0} shows a graph with a $P_5$ and a vertex in $core(G)\cap V^0$. Hence our result is tight for the class of cographs. Also we remark that this graph is bipartite.

\begin{figure}[htbp]
\begin{center}
 \includegraphics[width=5cm, height=4cm, keepaspectratio=true]{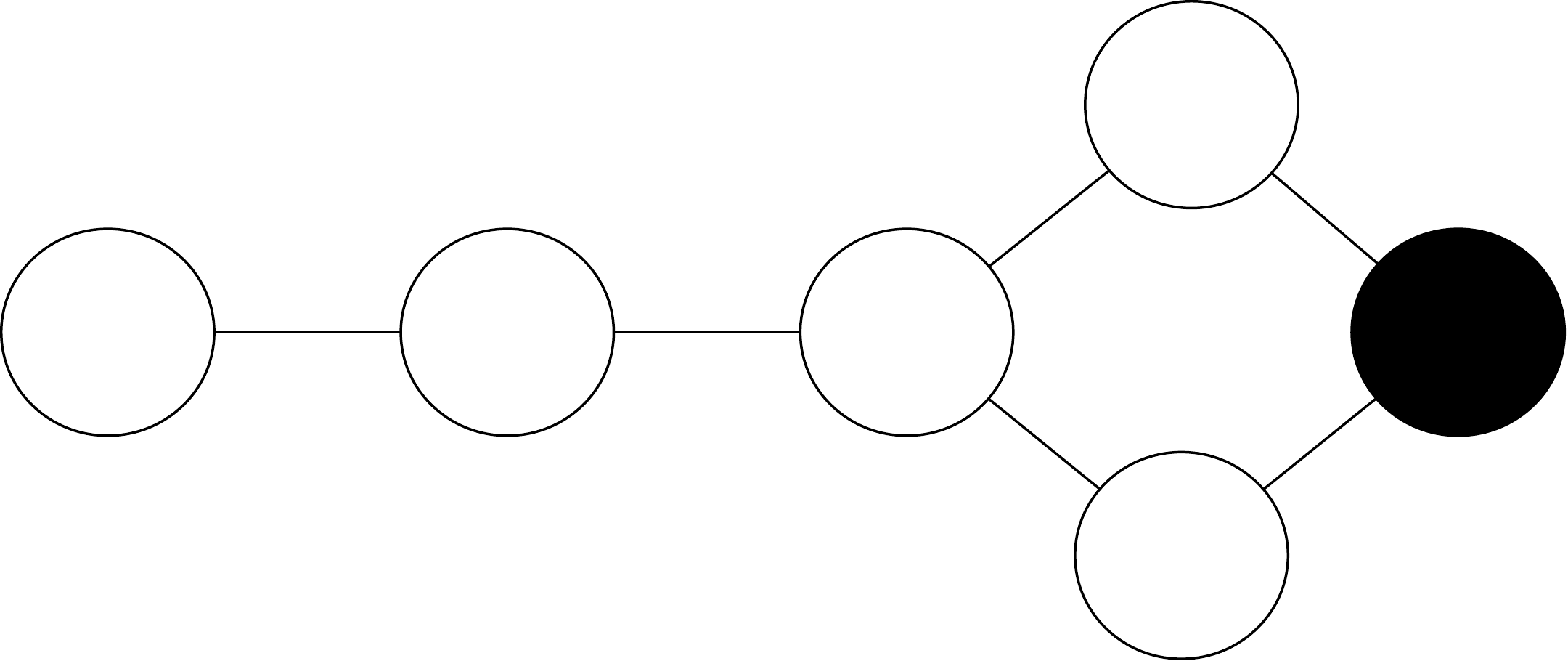}
\end{center}
\caption{The black vertex is in $core(G)\cap V^0$.}
\label{coreP5V0}
\end{figure}

\subsection{Claw-free graphs}\label{clawfree}
We are interested in connected claw-free graphs with at least two vertices. Given $H$  we  study the class of $(H,K_{1,3})-free$ graphs. We show that vertices $v$ with $v\in core(G)\cap V^+$ or $v\in core(G)\cap V^0$, may occurs when $H\in\{net, P_7\}$. In the case where $H\in\{bull, P_6\}$, or $H$ is a subgraph of the bull or $P_6$, we show that $v\in core(G)\cap V^0$ is not possible but there exist some graphs with $v\in core(G)\cap V^+$.\\

We give here a property of claw-free graphs proved by  Allan and Laskar \cite{Allan} that we will use later in several proofs: When a graph $G$ is $K_{1,3}$-free then $\gamma(G)=i(G)$. So in claw-free graphs there exists a $\gamma$-set which is an independant set.\\

Also recall that the class of line graphs is a subclass of claw-free graphs. Moreover if $G$ is diamond and odd hole free then $G$ is the line graph of a bipartite graph (see \cite{GraphCla}).
The Figure \ref{coreclawfree} shows a graph $G$ which is the line graph of a bipartite graph. Also $G$ is perfect and $(K_{1,3},K_4,net,diamond)$-free and $G$  has a vertex $v\in core(G)\cap V^0$.\\
 \begin{figure}[htbp]
\begin{center}
 \includegraphics[width=5cm, height=4cm, keepaspectratio=true]{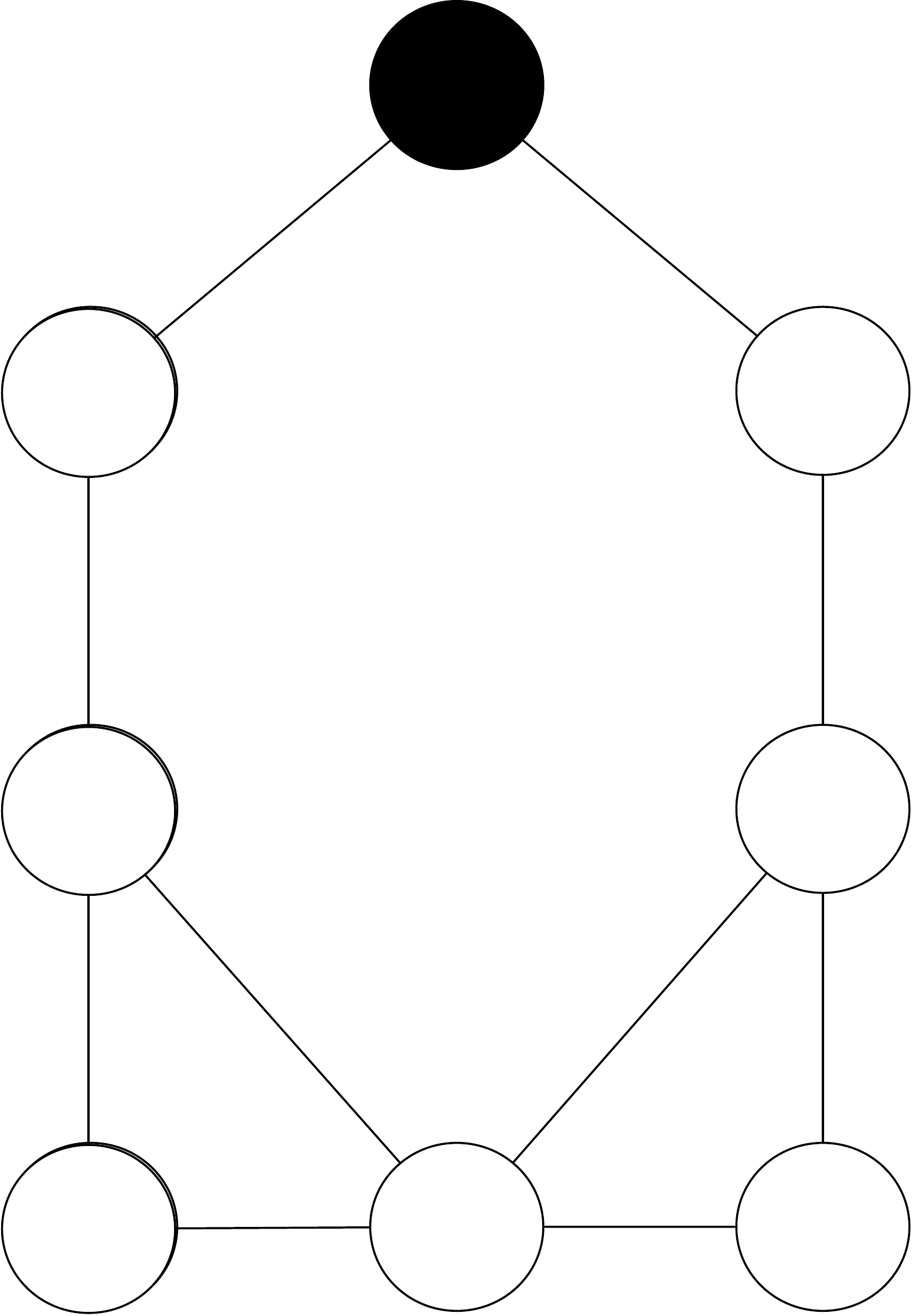}
\end{center}
\caption{The black vertex is in $core(G)\cap V^0$.}
\label{coreclawfree}
\end{figure}

We now give our results for the classes of $(claw,P_6)$-free and $(claw,bull)$-free graphs.

\begin{prop}\label{clawP6freecore}
Let $G=(V,E)$ be a connected $(claw,P_6)$-free graph with at least two vertices. $v\in core(G)$ if and only if $\gamma(G-v)>\gamma(G)$.
\end{prop}
\begin{proof}
From Theorem \ref{vincore} if  $\gamma(G-v)>\gamma(G)$ then $v\in core(G)$. Now let $v\in core(G)$.  Suppose Theorem \ref{vincore}.3: $\gamma(G-v)=\gamma(G)$ and every subset $S,\vert S\vert=\gamma(G),$ that dominates $G-v$ is such that $S\cap N[v]=\emptyset$.\\

From Fact \ref{cliqueneigh} $N(v)$ is not a clique. Since $G$ is claw-free we have $\alpha(N (v))=2$. For any independant $\gamma$-set $S$ it exists $a,b$ two private neighbors of $v$ such that $ab$ is a non-edge.\\

If $N(v)$ is not connected it consists of two anti-complete cliques $K_a$ and $K_b$. If $N(v)$ is connected then $P_k$ a maximal induced path is either $P_k=P_3$ or $P_k=P_4$. In the case $P_k=P_3$ then $N(v)$ consists of three cliques $K_a,K_b,K_{a'}$ such that $K_b$ is complete to $K_a$ and $K_{a'}$, and $K_a, K_{a'}$ are anti-complete. In the case $P_k=P_4$ then $N(v)$ consists of four cliques $K_a,K_b,K_{b'},K_{a'}$ with $K_b$ complete to $K_a$ and $K_{b'}$, $K_{b'}$ complete to $K_{a'}$ and $K_b$, the other cliques are pairwise anti-complete.

Let $u,w\in N_2(v)$ with a common neighbors $s\in K_p,p\in\{a,a',b,b'\}$. If $uw\not\in E$ then $G[\{u,w,s,v\}]$ is a claw. Thus all the vertices $u\in N_2(v)$ with at least a common neighbor $s, s\in N(v),$ induce a clique. With the same argument taking $u,w\in N_{p+1}(v),s\in N_p(v)$ and $v'\in N_{p-1}(v)\cap N(s)$ we have that for any $p\ge 1$ all the vertices $u\in N_{p+1}(v)$ with at least a common neighbor $s, s\in N_p(v),$ induce a clique.\\

We define a {\it twin clique partition} of $G$, $TCP(G)$ for short, as a partition $(K_0=\{v\}, K_1,\ldots,K_q)$ of $V$ such that each $K_i$ induces a clique, for every pair of $u,w\in K_i$ the vertices $u$ and $w$ are two false twins, i.e., $N[u]=N[w]$, and $K_i$ is inclusion-wise maximal, $1\le i\le q$. Note that $TCP(G)$ is unique and for any pair $K_i,K_j$ of  $TCP(G)$ either $K_i$ and $K_j$ are complete or anti-complete.\\

From $TCP(G)$ we define its reduced graph $H=(K,F)$ as follows. Its vertex set is $K=\{K_0=\{v\}, K_1,\ldots,K_q\}$ and two vertices of $K$ are adjacent if and only if their corresponding cliques in $TCP(G)$ are complete.

Note that since $G$ is (claw, $P_6$)-free $H$ is (claw,$P_6$)-free too. Moreover no pair $u,w$ of vertices of $H$ is such that $N[u]=N[w]$.\\

We give some relations between $G$ and $H$.
Note that for any $\gamma$-set $S$ of $G$ we have $\vert K_i\cap S\vert\le 1,0\le i\le q$. Also $\gamma(G)=\gamma(H)$. Moreover we have.
\begin{itemize}
\item $\vert K_i\cap S\vert= 1$ for any $\gamma$-set of $G$ if and only if $K_i\in core(H)$;
\item $\vert K_i\cap S\vert= 0$ for any $\gamma$-set of $G$ if and only if $K_i\in anticore(H)$.
\end{itemize}

Since $G$ and $H$ are equivalent relatively to the minimum dominating sets in the remaining of the proof we will write $G$ instead of $H$.\\

From Theorem \ref{vincore} we have that any $u\in N(v)$ has a neighbor in $N_2(v)$. Let us denote by $N_k(v,w)$ the set of vertices that are neighbor of $w$ and at distance $k$ of $v$.\\

Let $c\in N_2(v,a)\cap N_2(v,b)$.  $a,b$ are two private neighbors of $v$ relatively to $S$ so $c\not\in S$. Hence it exists $d\in S$ that dominates $c$, but $da, db$ are non-edges  and $G[\{a,b,c,d\}]$ is a claw. Thus $ N_2(v,a)\cap N_2(v,b)=\emptyset$.\\

Let $S'$ be an independant  $\gamma$-set of $G-v$. In $S'$ $a,b$ are dominated, respectively, by $\alpha \in N_2(v,a),\beta \in N_2(v,b),\alpha\ne \beta,$ and $\alpha\beta$ is a non-edge. Since $a,b$ are two private neighbors of $v$ relatively to $S$ we have that $\alpha,\beta \not \in S$ so it exists $\delta\in S$ that dominates $\alpha$ in $S$. If  $\delta\beta$ is a non-edge then $\delta-\alpha-a-v-b-\beta$ is a $P_6$. So $\delta\beta$ is an edge. Hence $G[\{v,a,\alpha,\delta,\beta,b\}]=C_6$. \\

Let the induced $C_6$ of $G$ be $C_6=\{v_1,v_2,v_3,w_1,w_2,w_3\}$. If $V=C_6$ then $core(G)=\emptyset$. So it exists  $w\in V-C_6$ and an edge $wu,u\in C_6$ since $G$ is connected. W.l.o.g let $u=v_1$. If $wv_2,ww_3$ are two non-edges then $G[\{w,v_1,v_2,w_1\}]$ is a claw. 
So, w.l.o.g., we can suppose that $wv_2$ is an edge.
 If $w$ has exactly two neighbors, $v_1,v_2$, in $C_6$ then $w-v_2-v_3-w_1-w_2-w_3$ is a $P_6$. If $w$ has exactly three neighbors in $C_6$ they must be successive else there is a claw. If $w$ has exactly four neighbors in $C_6$ either they are successive or they consist of two pairs of adjacent vertices separated by one vertex (at left and at right), else there is a claw. If $w$ has at least five neighbors in $C_6$ then there is a claw. Hence the vertices at distance one from $C_6$ are partioned into three sets: 
 \begin{description}
\item $W_3=\{w: w \textrm{ has exactly three successive neighbors in }C_6\}$;
\item  $W_4=\{w: w \textrm{ has exactly four successive neighbors in }C_6\}$;
\item  $W_2=\{w: w \textrm{ has two pairs of two successive neighbors in }C_6\}$.
\end{description}
Let $W=W_3\cup W_4\cup W_2$.

Let $u$ be a vertex at distance two from $C_6$. $u$ is adjacent to $w\in W$. Now since each $w\in W$ has two non-adjacent neighbors in $C_6$, $G$ has a claw. Hence $V=C_6\cup W,W\not=\emptyset$.\\

$\gamma_1=\{v_1,w_1\},\gamma_2=\{v_2,w_2\},\gamma_3=\{v_3,w_3\}$ are three distinct $\gamma$-sets of $C_6$.\\

When $W_2=\emptyset $, $\gamma_1$ and $\gamma_2$ are two distinct $\gamma$-sets of $G$ and $core(G)=\emptyset$.\\

Let $w\in W_2$. W.l.o.g. let $v_1,v_2,w_1,w_2$ be the neighbors of $w$. Suppose that it exists $w'\in W_2$ with the same neighbors as $w$. Since $G$ is a Twin Clique Partition $ww'$ must be a non-edge.  But $G[\{v_2,v_3,w,w'\}]$ is a claw.\\

It follows that $\vert W_2\vert\le 3$. Suppose that $W_2=\{w\}$. Let $v_1,v_2,w_1,w_2$ be the neighbors of $w$. Since each vertex $u\in W_3\cup W_4$ has a neighbor in $\gamma_1=\{v_1,w_1\}$ and another neighbor  in $\gamma_2=\{v_2,w_2\}$, $\gamma_1$ and $\gamma_2$ are two disjoint $\gamma$-sets of $G$ and $core(G)=\emptyset$.\\

Now let $W_2=\{w_{12},w_{13}\}$. W.l.o.g. let $v_1,v_2,w_1,w_2$ be the neighbors of $w_{12}$ and $v_1,v_3,w_1,w_3$ be the neighbors of $w_{13}$. Let $u\in W_3$. First suppose that the neighbors of $u$ are $v_1,v_2,v_3$. Then $G[\{v_1,w_3,w_{12},u\}]$ is a claw. By symmetry  there is one remaining case when the neighbors of $u$ are $w_3,v_1,v_2$. Then $G[\{v_2,u,w_{12},v_3\}]$ is a claw. Hence $W_3=\emptyset$. Let $u\in W_4$. First suppose that the neighbors of $u$ are $v_1,v_2,v_3,w_1$. Then $G[\{w_1,u,w_{13},w_2\}]$ is a claw. By symmetry the remaining case is when the neighbors of $u$ are $v_2,v_3,w_1,w_2$ but then  $G[\{w_2,u,w_{12},w_3\}]$ is a claw. Thus $W_4=\emptyset$ and $V=C_6\cup W_2$. Now $\gamma_1=\{v_1,w_1\}$ and $W_2=\{w_{12},w_{13}\}$ are two disjoint $\gamma$-sets of $G$. Hence $core(G)=\emptyset$.\\

Let $W_2=\{w_{12},w_{13},w_{23}\}$. Let $v_1,v_2,w_1,w_2$ be the neighbors of $w_{12}$, let $v_1,v_3,w_1,w_3$ be the neighbors of $w_{13}$ and let $v_2,v_3,w_2,w_3$ be the neighbors of $w_{23}$. As above $W_3\cup W_4=\emptyset$ (each vertex $u\in W_3\cup W_4$ induces a claw). Hence $V=C_6\cup W_2$. Then a $\gamma$-set of $G$ has size three. Now $\{v_1,w_2,v_3\}$ and $\{w_1,v_2,w_3\}$ are two disjoint $\gamma$-sets of $G$ and $core(G)=\emptyset$.\\

Hence the Twin Clique Partition of $G$ contains no vertex in $core(TCP(G))\cap V^0$ so $G$ as no vertex in $core(G)\cap V^0$.
\end{proof}

The Figure \ref{clawfreeP6} shows a $(claw,P_7)$-free graph $G$  with a vertex $v\in core(G)\cap V^0$. Since $G$ contains $P_6$ as an induced subgraph our result is tight.\\
 \begin{figure}[htbp]
\begin{center}
 \includegraphics[width=5cm, height=4cm, keepaspectratio=true]{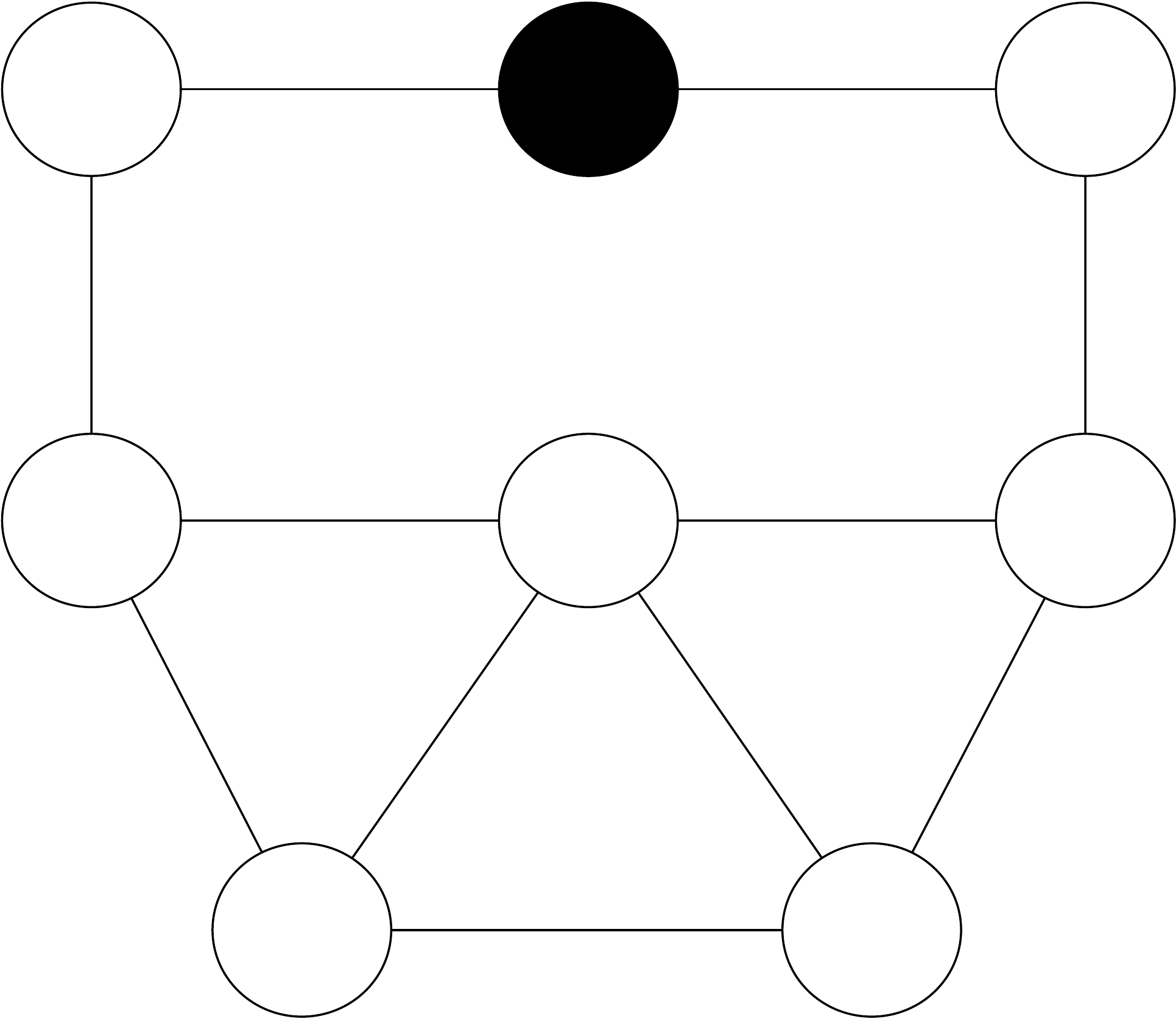}
\end{center}
\caption{A $(claw,P_7)$-free graph $G$. The black vertex is in $core(G)\cap V^0$.}
\label{clawfreeP6}
\end{figure}

\begin{prop}\label{clawbullfreecore}
Let $G=(V,E)$ be a connected $(claw,bull)$-free graph with at least two vertices. $v\in core(G)$ if and only if $\gamma(G-v)>\gamma(G)$.
\end{prop}
\begin{proof}
The beginning of the proof if the same as for Property \ref{clawP6freecore}'s proof. In $G$ (corresponding to the reduced graph of  $TCP(G)$)  $\alpha-a-v-b-\beta$ is an induced path $P$ and it exists $\delta,\delta\not\in P$ which is a neighbor of $\alpha$ and which is not a neighbor of $v,a,b$.  Hence from Lemma \ref{connect} there is an induced cycle $C_k,k\ge 6,$ that contains $v$.

Let the set of vertices of $C_k$  be $C_k=\{v_1,v_2,\ldots,v_k\}$. If $V=C_k$ then $core(G)=\emptyset$. 
So it exists  $w\in V-C_k$ and an edge $wu,u\in C_k$. W.l.o.g let $u=v_1$.
If $wv_2,wv_k$ are two non-edges then $G[\{w,v_1,v_2,v_k\}]$ is a claw.  If $w$ has five neighbors in $C_k$ the $G$ has a claw. If $w$ has exactly two (successive) neighbors in $C_k$, says $v_1,v_2$,  then $G[\{w,v_1,v_2,v_3,v_k\}]$ is a bull. If $w$ has exactly four  neighbors in $C_k$, says $v_1,v_2,v_3,v_4$  (from above they are successive), then $G[\{w,v_1,v_3,v_4,v_5\}]$ is a bull. Hence $w$ has exactly three successive neighbors in $C_k$, says $v_1,v_2,v_3$.

Let $u,u\not\in C_k$ be a neighbor of $w$. If $u$ has no neighbor in $C_k$ then $G[\{w,u,v_1,v_3\}]$ is a claw. Thus $u$ has three consecutive neighbors in $C_k$. Since $G$ is the reduced graph of $TCP(G)$ we have $N[w]\ne N[v_2]$ so, w.l.o.g., $u$ exists and we can suppose that $u\not\in N(v_2)$. Since $k\ge 6$ we have  $\vert N(w)\cap N(u)\vert\le 1$. If $N(w)\cap N(u)=\emptyset$ then $G[\{w,u,v_1,v_3\}]$ is a claw. So we can suppose that $N(w)\cap N(u)=\{v_1\}$ but then $G[\{w,u,v_1,v_3,v_{k-1}\}]$ is a bull.

Hence if $v\in core(G)$ then $v\not\in V^0$.
\end{proof}

One can remark that the graph in Figure \ref{clawfreeP6} is claw-free but contains a bull. It has a vertex in $core(G)\cap V^0$. Hence our result is tight.\\

Putting the properties together we have the following.
\begin{coro}
If $H$  is an induced subgraph of $P_6$ or of the bull  then for $(H,K_{1,3})$-free connected graphs with at least two vertices then  $v\in core(G)$ implies that $v\in V^+$. Moreover there exist graphs containing a bull or a $P_6$ with a vertex in $core(G)\cap V^0$.
\end{coro}

\subsection{Bipartite graphs}
The bipartite graphs are the graphs that are odd cycle free. The graph classes we studied above don't intersect the class of bipartite because their graphs can contain triangles. Trivially a connected bipartite claw-free graphs $G$ with at least two vertices is either a path or an even cycle. When $G$ is a path then $core(G)=V^+$. If $G$ is a cycle, $G$ is bipartite $2$-regular, then $core(G)=\emptyset$. 

Contrary to the $2$-regular bipartite graphs, cubic ($3$-regular) bipartite graphs may have vertices in $core(G)$.
Figure \ref{BipCubV0} shows a cubic  bipartite graph with a vertex in $ core(G)\cap V^0$. 
 \begin{figure}[htbp]
\begin{center}
 \includegraphics[width=8cm, height=5cm, keepaspectratio=true]{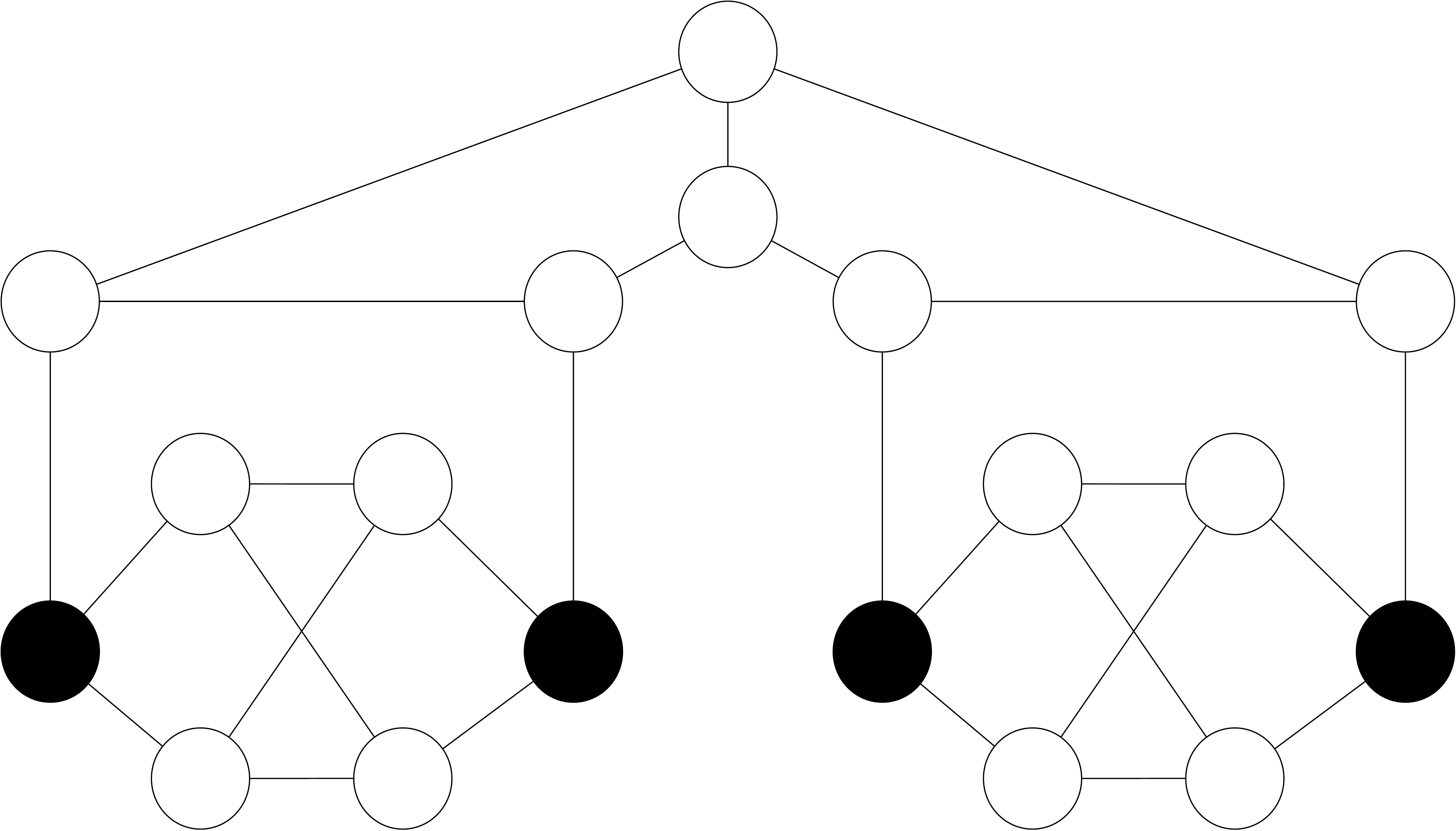}
\end{center}
\caption{A bipartite cubic graph with $\gamma(G)=5$. The black vertices are in $core(G)\cap V^0$.}
\label{BipCubV0}
\end{figure}

\section{Particular partitions of the vertex-set}
Relatively to the potential partitions of $V$ into $core(G),corona(G)-core(G),anticore(G)$ and $V^0,V^+$, we give connected graphs (without isolated vertices) whose vertex-set correspond to a specific partition. Two of them answer to two open questions by V. Samodivkin in \cite{Samodivkin}.

 $G=(V,E)$ a connected graph with at least two vertices for which $V=V^-$ is given in \cite{DomBook} (p. 139, Fig. 5.2). 
 
 Graphs such that $V=V^0$ are characterized in \cite{DomBook} (P.147, Theorem 5.23): these graphs must be such that $core (G)=\emptyset$, the complete bipartite  graph $K_{3,3}$ is one of them. Authors showed that graphs with $core (G)\ne\emptyset$ can exist but no such graph is exhibited. Finding such a graph correspond to the first question in the following article. In \cite{Samodivkin} V. Samodivkin raise the two following open questions.
\begin{enumerate}
\item Does there exists $G=(V,E)$ a connected graph  such that $\gamma(G-v)=\gamma(G)$ for all $v\in V$ and there is $u\in V, u\in core(G)$ ?
\item Does there exists $G=(V,E)$ a connected graph  such that there exists $w\in V, w\in V^+$,  all $v\in V, v\ne w,$ is such that $\gamma(G-v)=\gamma(G)$ and there is $u\in V,u\ne w, u\in core(G)$ ?
\end{enumerate}
The graph given in Figure \ref{CV0} give a positive answer to the first question.
 \begin{figure}[htbp]
\begin{center}
 \includegraphics[width=5cm, height=4cm, keepaspectratio=true]{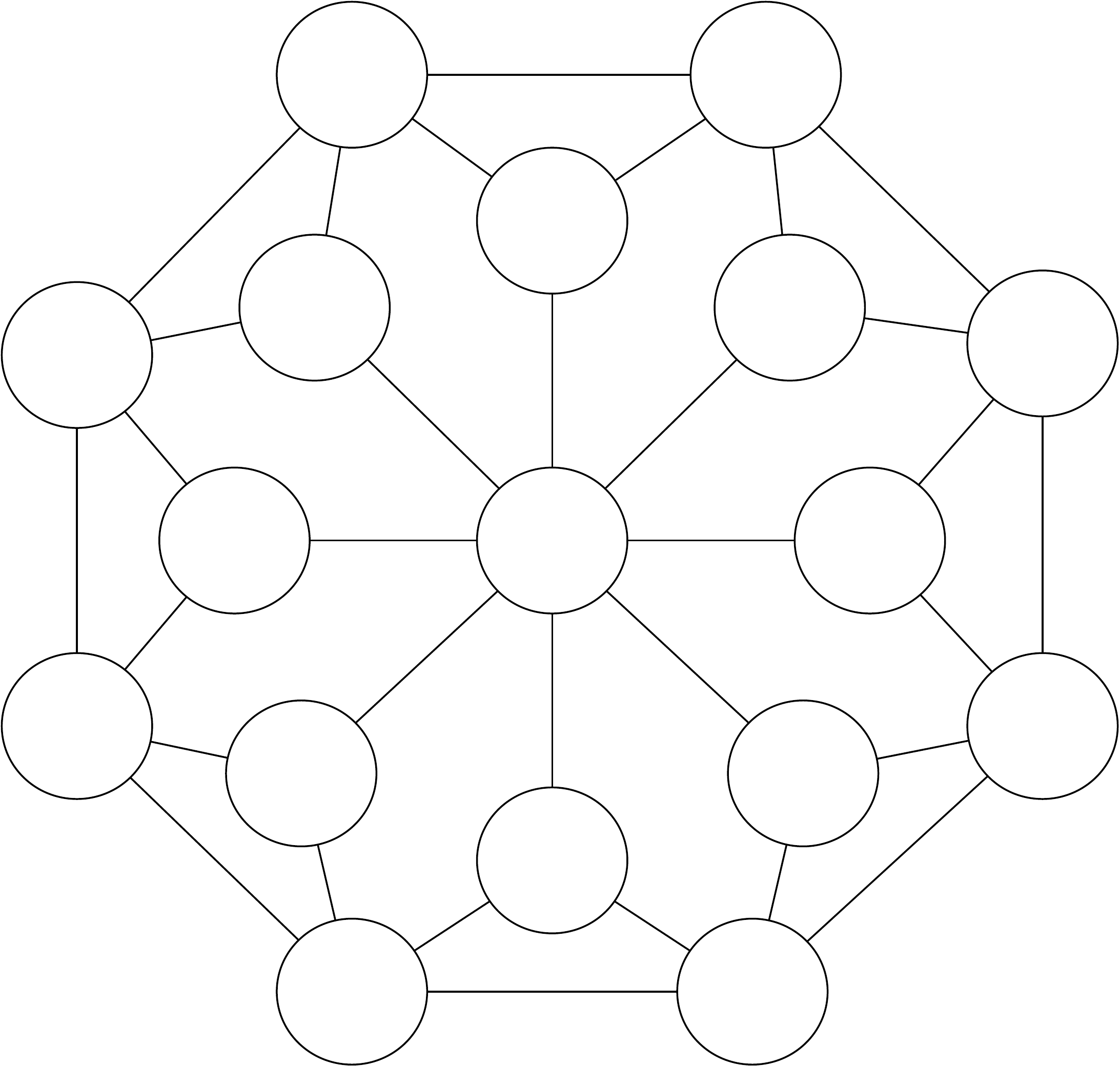}
\end{center}
\caption{All the vertices are in $V^0$, the central vertex is in $core(G)$.}
\label{CV0}
\end{figure}

The graph given in Figure \ref{CV+} give a positive answer to the second question.
 \begin{figure}[htbp]
\begin{center}
 \includegraphics[width=14cm, height=5cm, keepaspectratio=true]{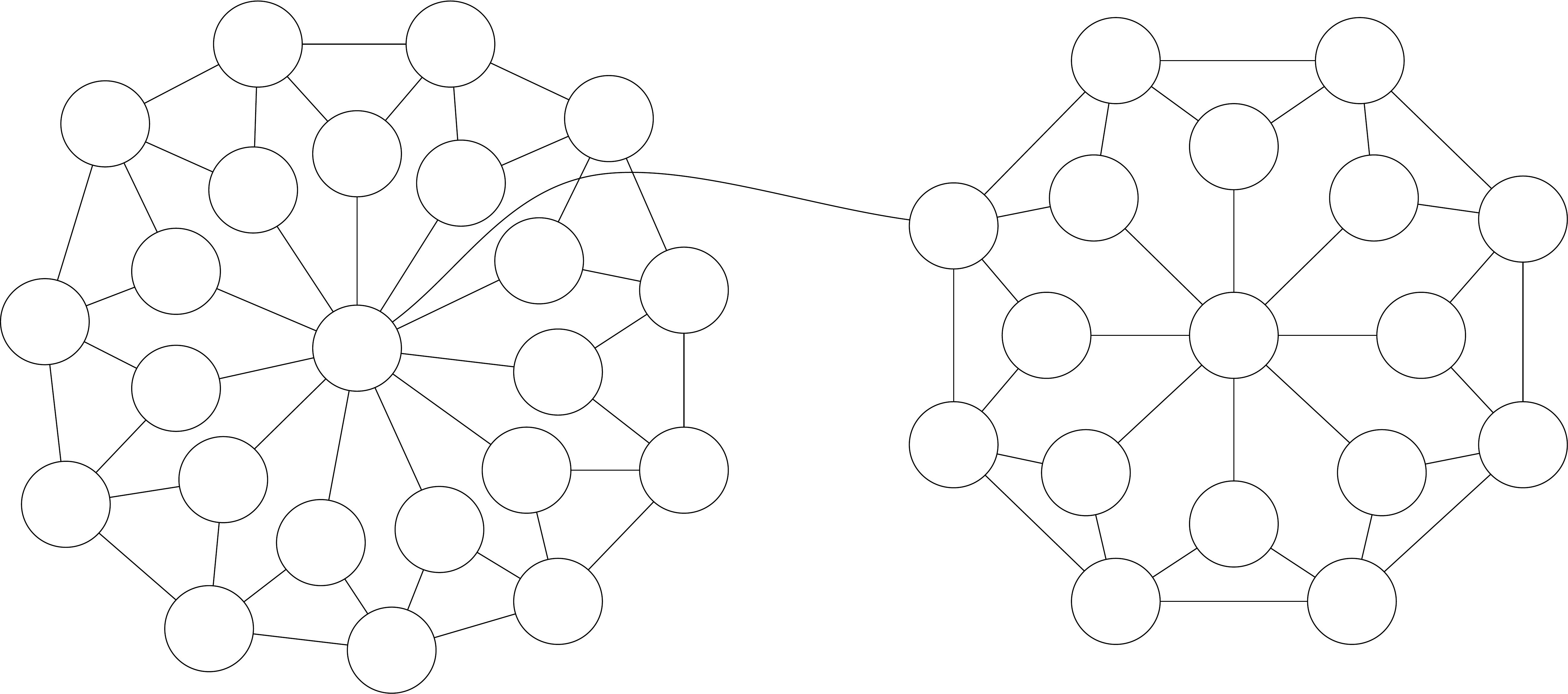}
\end{center}
\caption{The central vertex on left is in $V^+$, all the other vertices are in $V^0$, the central vertex on right is $V^0\cap core(G)$.}
\label{CV+}
\end{figure}

The graph given in Figure \ref{CV+0-} shows $G$ the graph  of minimum order such that $V^+,V^0,V^-\ne \emptyset$ and $anticore(G)=\emptyset$. The proof of its minimality is obtained by a computer.

 \begin{figure}[htbp]
\begin{center}
\includegraphics[width=5cm, height=4cm, keepaspectratio=true]{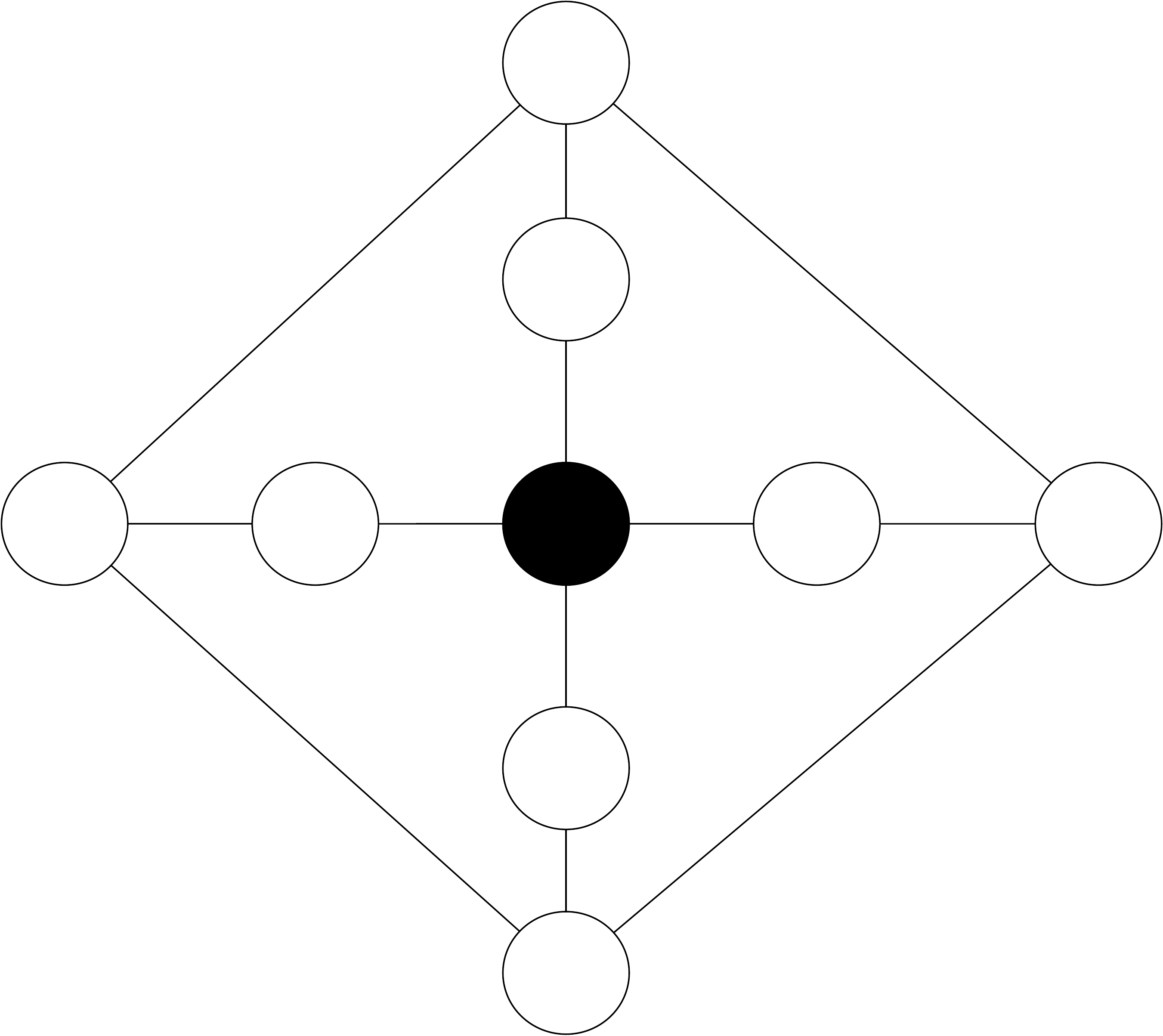}
\end{center}
\caption{The black vertex is in $V^+$, its four neighbors are in $V^0$, the four other vertices are in $V^-$.}
\label{CV+0-}
\end{figure}

The graph given in Figure \ref{V0Anti} is such that $V=(core(G)\cap V^0)\cup anticore(G)$.

\begin{figure}[htbp]
\begin{center}
\includegraphics[width=5cm, height=4cm, keepaspectratio=true]{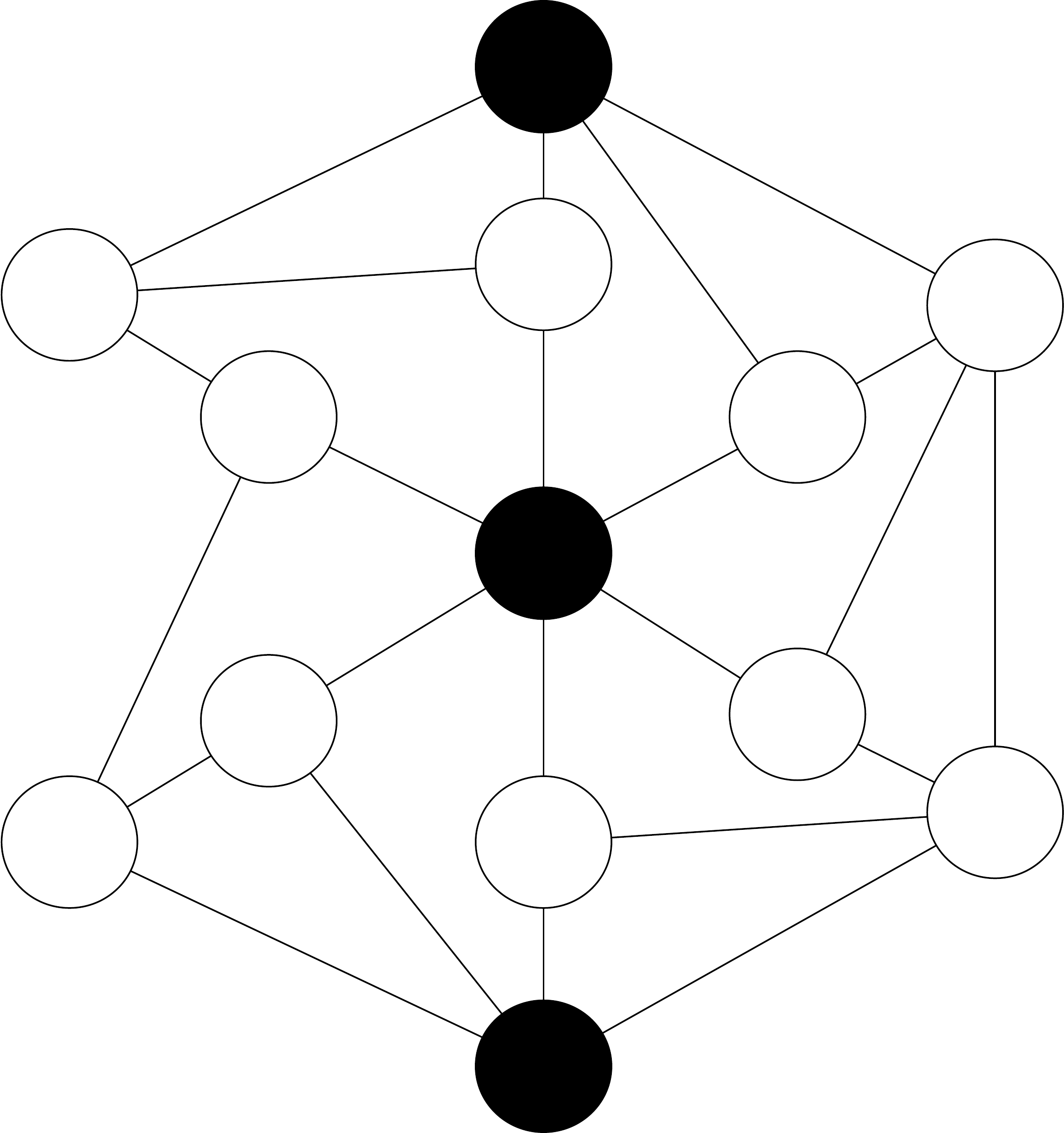}
\end{center}
\caption{A graph with $V=(core(G)\cap V^0)\cup anticore(G)$. The black vertices are in $core(G)\cap V^0$.}
\label{V0Anti}
\end{figure}

\section{Conclusion}
We gave a characterization for the vertices belonging  to all, none, or some minimum dominating sets in a graph. When the graph has no isolated vertices, the vertices belonging to all minimum dominating set are of two types.  The ones that increase de dominating number when suppressed from the graph, the ones for which the dominating number stays unchanged. For some subclasses of graphs we showed that only vertices of the first type may occur. Also we gave some graphs for which the partition of the vertex set omits some particular type of vertices, relatively to our characterization. 

Further directions of research may concern a {\it good} characterization of these particular classes of graphs. These studies include  the following question: For theses specific classes of graphs  the status of a given vertex can be decided in polynomial time?

\end{document}